\documentclass[11pt,hyp,]{nyjm}
\usepackage{amsmath, amsthm, amssymb, amsfonts,tikz,tikz-cd}
\usepackage{blindtext}
\usepackage{enumitem}
\usepackage{hyperref}
\hypersetup{nesting=true,debug=true,naturalnames=true}
\usepackage{graphicx,upref}

\theoremstyle{plain}
\newtheorem{theorem}{Theorem}[section]
\newtheorem{lemma}[theorem]{Lemma}
\newtheorem{corollary}[theorem]{Corollary}
\newtheorem{proposition}[theorem]{Proposition}

\theoremstyle{definition}
\newtheorem{definition}[theorem]{Definition}
\newtheorem{remark}[theorem]{Remark}
\newtheorem*{question*}{Question}

\newcommand{\integers}{\mathbb{Z}}
\newcommand{\ZZ}{\integers}

\newcommand{\homeo}{\operatorname{Homeo}}
\newcommand{\la}{\langle}
\newcommand{\ra}{\rangle}

\newcommand{\pmcg}{\operatorname{PMCG}}
\newcommand{\cpmcg}{\pmcg_c}
\newcommand{\mcg}{\operatorname{MCG}}

\newcommand{\flute}{S_{\mathrm{flute}}}

\newcommand{\dirlim}{\varinjlim}

\newcommand{\figref}[1]{\hyperlink{#1}{\ref*{fig:#1}}}

\newcommand{\mailurl}[1]{\email{\href{mailto:#1}{#1}}}

\title[Cohomology of MCG of genus one and zero surfaces]{First cohomology of pure mapping class groups of big genus one and zero surfaces}

\author{George Domat}
\address{(George Domat) Department of Mathematics, University of Utah, Salt Lake City, UT 84102, USA.}
\mailurl{domat@math.utah.edu}

\author{Paul Plummer}
\address{(Paul Plummer) Department of Mathematics, University of Oklahoma, Norman, OK 73019, USA.}
\mailurl{pplummer@math.ou.edu}

\thanks{Domat was partially supported by NSF DMS-1607236 amd NSF DMS-1840190.}
\thanks{Plummer was partially supported by NSF DMS-1651963 and NSF DMS-1611758}



\begin{document}

\begin{abstract}
 We prove that the first integral cohomology of pure mapping class groups of infinite type genus one surfaces is trivial. For genus zero surfaces we prove that not every homomorphism to $\ZZ$ factors through a sphere with finitely many punctures. In fact, we get an uncountable family of such maps. 
\end{abstract}
\maketitle
\tableofcontents


\section{Introduction}

Previous work in \cite{bigcohomology1} describes the first integral cohomology of pure mapping class groups of infinite type surfaces of genus at least 2 in terms of ends accumulated by genus. They prove that the cohomology of the compactly supported pure mapping class group is trivial and that the only other homomorphisms come from  handleshifts. In the finite type setting for surfaces without boundary it is known that once the genus is at least one the integral cohomology of the pure mapping class group is trivial. For a sphere with finitely many punctures the rank of the first cohomology is a function of the number of punctures. 
 
We investigate the integral cohomology in the infinite type genus one and zero cases. In the genus one case we prove the following.

\begin{theorem}\label{genus1}
Let $S$ be a genus one surface without boundary and with a nonempty closed set of marked points, representing punctures. Let $\pmcg(S)$ be the pure mapping class group of $S$. Then $H^1(\pmcg(S);\integers)=0$.
\end{theorem}

When $S$ is of finite type this is well known \cite{FM11}, so we will focus on the case when $S$ is not of finite type.

The idea of the proof is that if we have a homomorphism to $\ZZ$ which is not zero, then it is non-zero on some Dehn twist about a separating curve. In turn this will imply it is non-zero on a sequence of Dehn twists whose product converges in the group but whose image in $\ZZ$ will not converge. Work of \cite{bigcohomology1} and \cite{dudley1961} shows that any homomorphism to $\ZZ$ has to be continuous. We will use this sequence of Dehn twists to contradict the continuity of the homomorphism. We make use of the Gervais star presentation \cite{gervais01} to see how the homology groups of a finite type exhaustion fit together.

For surfaces of genus zero we have forgetful maps to finite type punctured spheres. The first integral cohomology of pure mapping class groups of spheres with at least four punctures is nontrivial so we cannot hope for a similar result as in Theorem \ref{genus1} for the genus zero case. However, we can ask: Do all homomorphisms from the pure mapping class group of an infinite type genus zero surface to $\ZZ$ factor through a forgetful map to a sphere with finitely many punctures? We answer the question negatively by constructing a specific homomorphism for the flute surface which does not factor as such. 
When $S$ is of infinite type and genus 0 there is always a forgetful map to the flute surface. The specifics of the construction then lead to the following theorem:
\begin{theorem}\label{genus0}
Let $S$ be a genus zero infinite type surface. Then the integral cohomology group $H^1( \pmcg (S); \integers)$ contains cohomology classes  which do not come from forgetful maps to finite type genus zero surfaces. Furthermore, there is an uncountable family of such classes.
\end{theorem}
The uncountability comes from being able to vary the construction to encode a Cantor set of cohomology classes. 

This section of the paper will rely on the fact that any finite type subsurface of a genus zero surface has mapping class group isomorphic to a braid group. \cite{Fabel06} contains some results and discussions on infinite stranded braid groups and its connections to the mapping class group of a disk with infinitely many punctures. 

Our results together with those from \cite{bigcohomology1} give an almost complete picture of the first integral cohomology of these big pure mapping class groups. The only piece still missing is an explicit description of the cohomology in the genus zero case. The cohomology groups break down into the following categories:
\begin{itemize}
    \item $S$ has more than one end accumulated by genus (\cite{bigcohomology1}): The rank of $H^{1}(\pmcg(S);\ZZ)$ is one less than the number of ends accumulated by genus if there are finitely many such ends and infinite if there are infinitely many. Furthermore, all non-trivial classes are not supported on finite type subsurfaces. 
    \item $S$ has at most one end accumulated by genus with $g>0$ (\cite{bigcohomology1} and Theorem \ref{genus1}): $H^{1}(\pmcg(S);\ZZ)$ is trivial. 
    \item $S$ is genus zero (Theorem \ref{genus0}): $H^{1}(\pmcg(S);\ZZ)$ has infinite rank and contains classes both supported on finite type subsurfaces and ones that are not. 
\end{itemize}

\begin{question*}
    What is an explicit description of $H^{1}(\pmcg(S);\ZZ)$ when $S$ is genus zero?
\end{question*}

Throughout the paper we will identify $H^1(G;\ZZ)$ with the set of homomorphisms $G \to \ZZ.$ We will also usually be considering punctures as marked points and conflating compact subsurfaces with finite type subsurfaces for the sake of convenience. 

\textbf{Acknowledgments:} We thank Mladen Bestvina and Jing Tao for numerous helpful conversations throughout this project. We also thank Javiar Aramayona, Priyam Patel, and Nick Vlamis for their interest and comments on an earlier draft. We thank the referree for numerous helpful comments. This work was started at the Fields Institute's Thematic Program on Teichm\"{u}ller Theory and its Connections to Geometry, Topology and Dynamics and as such we thank the institute and the organizers for their support. 

\section{Background}

\subsection{Big pure mapping class group}

Let $S$ be a connected, orientable, second-countable surface, possibly with boundary. Let $\homeo^{+}_{\partial}(S)$ be the group of orientation preserving homeomorphisms of $S$ which fix the boundary pointwise. The \textbf{mapping class group}, $\mcg(S)$, is defined to be 
\begin{align*}
    \mcg(S) = \homeo_{\partial}^{+}(S)/\sim
\end{align*}
where two homeomorphisms are equivalent if they are isotopic relative to the boundary of $S$. $\homeo_{\partial}^{+}(S)$ is equipped with the compact-open topology, which induces the quotient topology on $\mcg(S)$. Subgroups of $\mcg(S)$ come equipped with the subspace topology.

The \textbf{pure mapping class group}, $\pmcg(S)$, is defined to be the kernel of the action of $\mcg(S)$ on the space of ends of $S$. 

We say $f \in \mcg(S)$ is \textbf{compactly supported} if $f$ has a representative that is the identity outside of a compact set in $S$. We denote the subgroup of $\mcg(S)$ of compactly supported mapping classes as $\cpmcg(S)$. Note that any compactly supported mapping class is automatically in $\pmcg(S)$. Patel and Vlamis proved

\begin{theorem}[\cite{patelvlamis}]\label{closureofcompact}
$\overline{\cpmcg(S)} = \pmcg(S)$ if and only if $S$ has at most one end accumulated by genus. In particular, $\overline{\cpmcg(S)} = \pmcg(S)$ if $S$ is genus one or zero. 
\end{theorem}

We say that a sequence of curves $\{\alpha_{i}\}$ \textbf{leaves every compact set} if for every compact subset $K$ of $S$ there exists an $N>0$ such that $\alpha_{n}\cap K = \emptyset$ for all $n>N$. 

\subsection{Polish groups and automatic continuity}

\begin{definition}
    A \textbf{Polish group} is a topological group that is separable and completely metrizable as a topological space.
\end{definition}

\begin{lemma}[\cite{bigcohomology1}, Corollary 2.5]\label{closedsubgroupspolish}
The group $\mcg^\pm(S)$ is a Polish group with the compact-open topology. Hence $\pmcg(S)$ is a Polish group as it is a closed subgroup.
\end{lemma}

\begin{remark}
The above is using work in \cite{HMV17} and \cite{BDR17} which shows that the automorphism group of the curve graph of $S$ is isomorphic to $\mcg^\pm(S)$.
\end{remark}

This is important for our applications because of the following theorem of Dudley. Dudley's theorem is more general but we state the relevant version of the theorem to our work. 

\begin{theorem}[\cite{dudley1961}]\label{automaticcontinuity}
    Every homomorphism from a Polish group to $\ZZ$ is continuous.
\end{theorem}

We can use these results to obtain the following necessary condition for a map from $\pmcg(S)$ to $\ZZ$ to be a homomorphism. 

\begin{lemma}\label{compactlemma}
    For any surface $S$ and homomorphism $f\colon\pmcg(S) \rightarrow \ZZ$, $f$ cannot be non-zero on a sequence of Dehn twists about curves that leave every compact (or finite type) subsurface of $S$. In other words, there is a compact (or finite type) subsurface $K_{0} \subset S$ such that if $f(T_{\alpha}) \neq 0$ then $\alpha \cap K_{0} \neq \emptyset$ where $T_{\alpha}$ is the Dehn twist about the curve $\alpha$. 
\end{lemma}
\begin{proof}
    The statement is trivial for finite type surfaces. By Lemma \ref{closedsubgroupspolish} and Theorem \ref{automaticcontinuity} $f$ is continuous. Suppose $f$ were nonzero on such a sequence of Dehn twists. Then we could find a sequence of Dehn twists $\{T_{\alpha_{i}}\}$ such that $f(T_{\alpha_{i}}) > 0$ for all $i$ (possibly taking inverses if all twists are negative) and the $\alpha_{i}$ leave every compact set. Then $\prod_{i=1}^{n} T_{\alpha_{i}}$ converges in $\pmcg(S)$ as $n$ goes to infinity. However, $f\left(\prod_{i=1}^{n}T_{\alpha_{i}}\right) = \sum_{i=1}^{n} f(T_{\alpha_{i}})$ does not converge in $\ZZ$, contradicting the continuity of $f$. 
    \end{proof}

We also get the following immediate consequence.

\begin{proposition}
    Let $S$ be an infinite type surface with at most one end accumulated by genus. Let $K_{0} \subset K_{1} \subset \cdots$ be an exhaustion of $S$ by finite type surfaces. These induce inclusions $\pmcg(K_{n}) \hookrightarrow \pmcg(K_{n+1})$ for all $n$. Then we have that $H^{1}(\pmcg(S);\ZZ)$ injects into the inverse limit $\varprojlim H^{1}(\pmcg(K_{n});\ZZ)$. 
\end{proposition}
\begin{proof}
    We get a map $H^{1}(\pmcg(S);\ZZ) \rightarrow \varprojlim H^{1}(\pmcg(K_{n});\ZZ)$ by restriction of a cohomology class to the subsurfaces $K_{n}$. Now suppose that $\phi \in H^{1}(\pmcg(S);\ZZ)$ restricts to the zero map on each $\pmcg(K_{n})$. Given any $f \in \pmcg(S)$ we can approximate $f$ by a sequence of compactly supported mapping classes by Theorem \ref{closureofcompact}. Thus by continuity of $\phi$ and Theorems \ref{closedsubgroupspolish} and \ref{automaticcontinuity}, we see that $\phi$ is the zero map on all of $\pmcg(S)$. 
\end{proof}

\section{Genus one}

\subsection{Pure mapping class group of finite type genus one surfaces}

Let $S_{1,n}^{b}$ be a surface of genus $1$ with $n$ punctures and $b$ boundary components. One of the main tools we use is a description of $\pmcg(S_{1,n}^b)$ in terms of Gervais star relations. 

For any subsurface of $S_{1,n}^{b}$ homeomorphic to $S_{1}^{3}$ and curves as in Figure \figref{star} we have the relation $(T_{c_{1}}T_{c_{2}}T_{c_{3}}T_{b})^{3} = T_{d_{1}}T_{d_{2}}T_{d_{3}}$, where $T_{a}$ denotes the Dehn twist about the curve $a$. This relation is called a \textbf{star relation}. We have degenerate star relations when one or more of the boundary curves, $d_{i}$, is null-homotopic in $S_{1,n}^{b}$, e.g., if $d_{3}$ is null-homotopic then the relation becomes $(T_{c_{1}}T_{c_{2}}T_{c_{2}}T_{b})^{3} = T_{d_{1}}T_{d_{2}}$. In fact, a presentation of the mapping class group can be defined using these star relations and braid relations. Such a presentation can be found in \cite{gervais01}.

\begin{figure}
	    \centering
	    \def\svgwidth{.5\columnwidth}
		    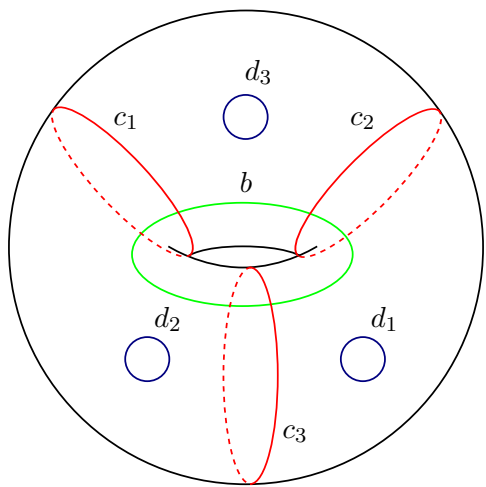
		    \caption{Curves in a Star Relation.}
	    \label{fig:star}
\end{figure}

We will also be using the first homology and cohomology of $\pmcg(S_{1,n}^b)$.

\begin{theorem}[\cite{korkmaz02}]\label{fin-homology}
$H_1(\pmcg(S_{1,n}^b);\ZZ) \cong \ZZ^b$ when $b>0$. Furthermore, we can choose a basis for the homology corresponding to $b-1$ Dehn twists about $b-1$ boundary components and a Dehn twist about a nonseparating curve. When $b=0$ we have $H_1(\pmcg(S_{1,n});\ZZ)=\ZZ/12 \ZZ $.
\end{theorem}

Note that by an application of the change of coordinates principle all Dehn twists about nonseparating curves are identified in homology. When the surface does not have boundary the homology can be derived directly from a degenerate star relation.  

\subsection{Homology of exhaustions}

Let $S$ be any orientable infinite type genus $1$ surface and let $F_{0} 
\subset F_{1} \subset F_{2} \subset \cdots$ be an exhaustion of $S$ by finite type subsurfaces with $F_{0}$ having genus one. Without loss of generality we may assume that $F_{n+1}$ is obtained from $F_{n}$ by gluing a disk with a puncture, an annulus with a puncture, or a pair of pants to some boundary curve of $F_{n}$. Let $P_{n} = \pmcg(F_{n})$. 

Note that $\cpmcg(S)= \dirlim P_{n}$, so any homomorphism 
\[
f\colon\cpmcg(S) \to \ZZ
\] 
restricts to a map $f_n\colon P_{n} \to \ZZ$ for all $n$. Conversely, we say that a sequence of maps $\{f_{n}:P_{n} \rightarrow \ZZ\}$ is \textbf{consistent} if $f_{n+1}\vert_{P_{n}} = f_{n}$ for all $n$. If we define consistent maps $f_n$ for each $n$, these determine a map $f$ in the limit.

A basis for $H_1(P_n;\ZZ)$ is given by $\langle \tau, \partial_1, \dots, \partial_{m-1} \rangle$ where $\tau$ represents a Dehn twist about a non-separating curve and $\partial_i$ represents a Dehn twist about a boundary curve $b_i$, where all the boundary curves are $b_1, \dots ,b_m$. We now describe how the exhaustion is realized in homology. 

\begin{lemma}\label{homology-relationship}
The map $H_{1}(P_{n};\ZZ) \rightarrow H_{1}(P_{n+1};\ZZ)$ takes one of the following three forms depending on how $F_{n+1}$ is obtained from $F_{n}$:
\begin{enumerate}
    \item (Disk with a Puncture);
    If we cap a boundary curve $b_{i}$ with a punctured disk then $\partial_{i}$ is killed and our new basis for $H_{1}(P_{n+1};\ZZ)$ is given by $\langle \tau, \partial_1, \dots, \partial_{i-1}, \partial_{i+1} \dots, \partial_{m-1} \rangle$.
    
    \item (Annulus with a Puncture):
    If we glue an annulus with a puncture to $b_{i}$ then we get a new boundary component $b_{i}'$ and $\partial_{i} = \partial_{i}'$; i.e. the two Dehn twists will be homologous. A new basis for $H_{1}(P_{n+1};\ZZ)$ is $\langle \tau,\partial_{1},\ldots,\partial_{i-1},\partial_{i}',\partial_{i+1},\ldots,\partial_{m-1}\rangle$.
    
    \item (Pair of Pants):
    If we glue a pair of pants to $b_{i}$ then we get two new boundary components, say $b_{i_{0}}$ and $b_{i_{1}}$, and $\partial_{i} = \partial_{i_{0}} + \partial_{i_{1}}$. The new basis for $H_{1}(P_{n+1};\ZZ)$ is $\langle \tau, \partial_1, \dots \partial_{i-1}, \partial_{i_{0}},\partial_{i_{1}}, \partial_{i+1}, \dots, \partial_{m-1} \rangle$
\end{enumerate}
\end{lemma}

\begin{proof}
    See Figure \figref{homexample} for examples of the annulus with a puncture and pair of pants cases.
    \begin{enumerate}
        \item 
        Suppose $F_{n+1}$ is obtained from $F_{n}$ by capping the boundary curve $b_{i}$ with a punctured disk. Then $b_{i}$ becomes null homotopic in $F_{n+1}$ so $\partial_{i}$ is now trivial. Furthermore, the rank of $H_{1}(P_{n+1};\ZZ)$ is decreased by one and $\langle \tau, \partial_1, \dots, \partial_{i-1}, \partial_{i+1} \dots \partial_{m-1} \rangle$ is a basis.
        \item
        Suppose $F_{n+1}$ is obtained from $F_{n}$ by gluing an annulus with a puncture to $b_{i}$. Let $b_{i}'$ be the new boundary curve and $\partial_{i}'$ the Dehn twist about it. Now we will apply a star relation to see that $\partial_{i} = \partial'_{i}$. Let $a$ be the curve on $F_{n+1}$ that bounds all boundary components and punctures other than $b_{i}'$ and the puncture we glued on. Then we have two degenerate star relations $12\tau = \alpha + \partial_{i}'$ and $12\tau = \alpha + \partial_{i}$ where $\alpha$ is the Dehn twist about $a$. Thus we see that $\partial_{i}' = \partial_{i}$ as desired. 
        \item
        Suppose $F_{n+1}$ is obtained from $F_{n}$ by gluing a pair of pants to $b_{i}$. Let $b_{i_{0}}$ and $b_{i_{1}}$ be the two new boundary curves and $\partial_{i_{0}}$ and $\partial_{i_{1}}$ the Dehn twists about them respectively. We proceed as in the previous case. Let $a$ be the curve on $F_{n+1}$ that bounds all punctures and boundary curves other than $b_{i_{0}}$ and $b_{i_{1}}$. Then we have one degenerate star relation $12\tau = \alpha + \partial_{i}$ and a star relation $12\tau = \alpha + \partial_{i_{0}} + \partial_{i_{1}}$. Thus we see that $\partial_{i} = \partial_{i_{0}} + \partial_{i_{1}}$. \qedhere
    \end{enumerate}    
    	\begin{figure}
	    \centering
	    \def\svgwidth{.8\columnwidth}
		    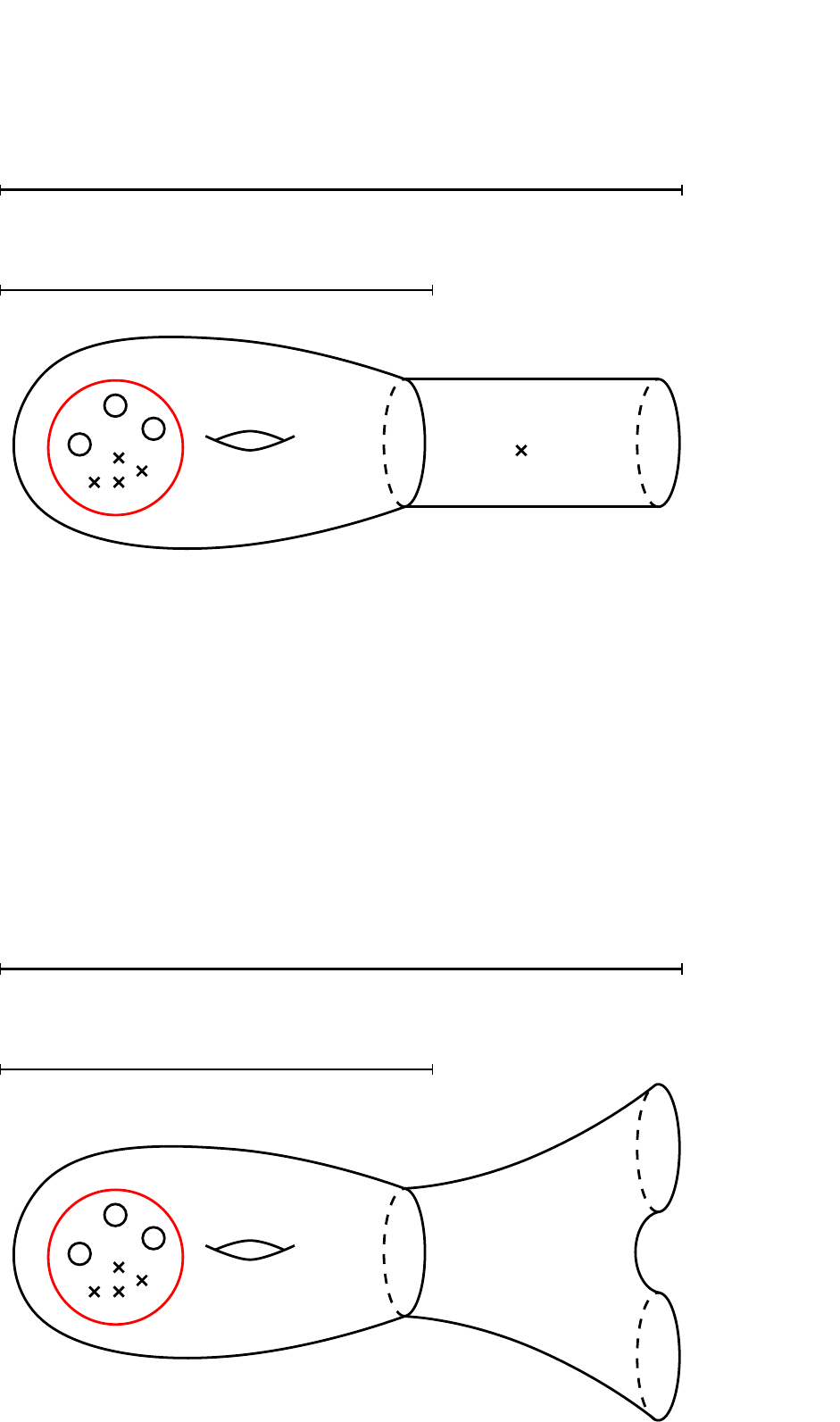
		    \caption{Example of the punctured annulus and pair of pants cases. The x's represent punctures.}
	    \label{fig:homexample}
	    \end{figure}
	    \end{proof}

We also need the following relation in the homology of these finite type surfaces. 

\begin{lemma}\label{boundaryrelation}
    If $F$ is a genus one surface of finite type with boundary curves $b_{1},\ldots,b_{m}$ then in $H_{1}(\pmcg(F);\ZZ)$ the relation $12\tau = \sum_{i=1}^{m} \partial_{i}$ holds, where $\partial_{i}$ is the Dehn twist about $b_{i}$.
\end{lemma}

\begin{proof}
    We first note that if $m=0$ we get $12\tau = 0$ from a fully degenerate star relation. Next we induct on the number of boundary curves. If $m=1$ this is an application of a doubly degenerate star relation. The case that $m=2$ is the same with a degenerate star relation. 
    
    Suppose $F$ has $m>2$ boundary curves, say $b_{1}\ldots,b_{m}$. Let $b'$ be the separating curve that isolates $b_{m}$ and $b_{m-1}$ from the rest of the surface. Now by the induction step $12\tau = \partial' + \sum_{i=1}^{m-2}\partial_{i}$ where $\partial'$ is the Dehn twist about $b'$. By the above lemma, $\partial' = \partial_{m-1}+\partial_{m}$, proving the desired result. 
    \end{proof}

\subsection{Proof of theorem \ref{genus1}}

\begin{proof}[Proof of theorem \ref{genus1}]

Given an exhaustion of $S$ by finite type surfaces, $F_{0}\subset F_{1} \subset F_{2} \subset \cdots$, as above, let $P_{n} = \pmcg(F_{n})$. We then obtain a directed system of groups and have $\varinjlim P_{n} = \cpmcg(S)$. Let $f\colon \pmcg(S) \rightarrow \ZZ$ be a homomorphism. Since $\pmcg(S)$ is a Polish group by Theorem \ref{closedsubgroupspolish}, $f$ is continuous, by Theorem \ref{automaticcontinuity}. We obtain a continuous map $f_{c}\colon \cpmcg(S) \rightarrow \ZZ$, by restriction, and hence maps $f_{n}\colon P_{n} \rightarrow \ZZ$ for each $n$. 

We claim that $f_{n} = 0$ for all $n$. Note that then $f_{c} = 0$ and since $\pmcg(S) = \overline{\cpmcg(S)}$, by Theorem \ref{closureofcompact}, we will also have $f=0$ and the proof would be complete. 

Suppose, to the contrary, that some $f_{n} \neq 0$. Since a basis for $H_{1}(P_{n},\ZZ)$ is $\la \tau, \partial_{1}, \ldots, \partial_{m-1} \ra$ as above, we must have that $f_{n}$ is nonzero on some basis element. In fact, we also have that $f_{n}$ is nonzero on at least one Dehn twist about some boundary curve. Indeed, if $f_{n}(\partial_{i}) = 0$ for $i = 1,\ldots,m-1$ and $f_{n}(\tau) \neq 0$, then using the relation that $12\tau = \sum_{i=1}^{m} \partial_{i}$ we must have that $f_{n}(\partial_{m}) \neq 0$. Let $\hat{\partial}_{n}$ be a Dehn twist about a boundary curve, say $\hat b_{n}$, such that $f_{n}(\hat{\partial_{n}}) > 0$. We can assume it is positive by possibly taking an inverse Dehn twist. 

Next we examine $f_{k}\colon P_{k} \rightarrow \ZZ$ where $k > n$ is the next time that an element of our exhaustion is obtained by gluing a punctured disk, a punctured annulus, or a pair of pants to the boundary curve $\hat{b}_{n}$. Note that we cannot have that $F_{k}$ is obtained by gluing a punctured disk to $\hat{b}_{n}$ since then $\hat\partial_{n}$ would become trivial and $f_{k}(\hat\partial_{n})=0$, contradicting the compatibility of the $f_{n}$ with our directed system.

Thus we have that $F_{k}$ is obtained by either gluing a punctured annulus or a pair of pants to $\hat b_{n}$. Applying Lemma \ref{homology-relationship} we then find a new Dehn twist about a boundary curve $\hat\partial_{k}$ such that $f_{k}(\hat{\partial}_{k})>0$. Note that we still have $f_{k}(\hat{\partial}_{n}) > 0$ as well. 

We continue this process and obtain a sequence of Dehn twists $\hat{\partial}_{n_{i}}$ with the property that $f_{k}(\hat{\partial}_{n_{i}}) > 0$ for all $n_{i} < k$. Since the maps $f_{k}$ are a compatible sequence of maps and are compatible with $f$, there exists an infinite sequence $\{\hat{\partial}_{n_{i}}\}$ leaving every compact set such that $f(\hat{\partial}_{n_{i}}) \neq 0$, contradicting Lemma \ref{compactlemma}.
\end{proof}

Note that $\cpmcg(S)$ will have nontrivial cohomology, but it just does not extend (continuously) to $\pmcg(S)$.

\begin{remark}
In this setting, $\pmcg(S)$ does have a finite index subgroup which surjects onto $F_2$, the free group on two generators. This is because we have a surjective group homomorphism, coming from a forgetful map, to $\pmcg(S_{1,1})$. It is a classical result that $\pmcg(S_{1,1})$ is virtually free, as it is isomorphic to $SL(2,\integers)$. Taking the preimage of that finite index free subgroup we get the desired finite index subgroup.
\end{remark}

\section{Genus zero}

Let $S$ be any surface, let $X \subset S$ be a compact and totally disconnected collection of points in the interior of $S$, and let $S_{X}$ denote the surface obtained from $S$ by marking all the points in $X$. Now there is a natural homomorphism, called the \textbf{forgetful map}, $\mathcal{F}:\pmcg(S_{X}) \rightarrow \pmcg(S)$ realized by forgetting that the points in $X$ are marked. The goal of this section is to show that there are homomorphisms to $\ZZ$ from the pure mapping class group in the genus zero case which \emph{do not} factor through forgetful maps to finite type surfaces. 

Let $\flute$ be a sphere with infinitely many isolated punctures with one accumulation point. We will now build a cohomology class on $\pmcg(\flute)$ which ``sees" every puncture and thus cannot factor through a forgetful map. We first build a homomorphism from $\cpmcg(\flute)$ to $\ZZ$ and then  show that it extends to the closure and hence the entire pure mapping class group. Let $K_{n}$ be a finite type exhaustion of $\flute$ where $K_{0}$ is a disk with two punctures and then $K_{n+1}$ is obtained from $K_{n}$ by gluing an annulus with one puncture to the boundary of $K_{n}$. 

We have that the pure mapping class group of each surface $K_{n}$ is the pure braid group $PB_{n+2}$. $H^{1}(PB_{n};\ZZ)$ has a basis coming from Dehn twists about each pair of punctures. The Dehn twist about the boundary curve can be written as a product of all these Dehn twists about pairs of punctures, without inverses \cite{FM11}. Let $a$ and $b$ be the two punctures in $K_{0}$ and enumerate the other punctures of $K_{n}$ by $1,\ldots, n$. Then the curve about a pair of punctures is denoted $\gamma_{ij}$ where $i,j>0$ or $= a,b$ and the twist about such a curve is $T_{\gamma_{ij}}$. See Figure \figref{flute} for an example of these curves. We define $\phi_{0}\colon\pmcg(K_{0}) \rightarrow \ZZ$ to be the zero map; i.e., $\phi$ sends the only basis element, $T_{\gamma_{ab}}$, to $0$. 

Next $\phi_{1}\colon \pmcg(K_{1}) \rightarrow \ZZ$ is defined on the basis by sending
\begin{align*}
    T_{\gamma_{ab}} &\mapsto 0 \\
    T_{\gamma_{a1}} &\mapsto 1 \\
    T_{\gamma_{b1}} &\mapsto -1.
\end{align*}
Note that $\phi_{1}\vert_{\pmcg(K_{0})} = \phi_{0}$. We now define $\phi_{n+1}$ from $\phi_{n}$ recursively by setting $\phi_{n+1}\vert_{\pmcg(K_{n})} = \phi_{n}$ and on the new basis elements
\begin{align*}
    T_{\gamma_{i(n+1)}} & \mapsto 0 \;\;\; \text{ if } i\neq a,b \\
    T_{\gamma_{a(n+1)}} &\mapsto 1 \\
    T_{\gamma_{b(n+1)}} &\mapsto -1.
\end{align*}

\begin{figure}
	    \centering
	    \def\svgwidth{\columnwidth}
		    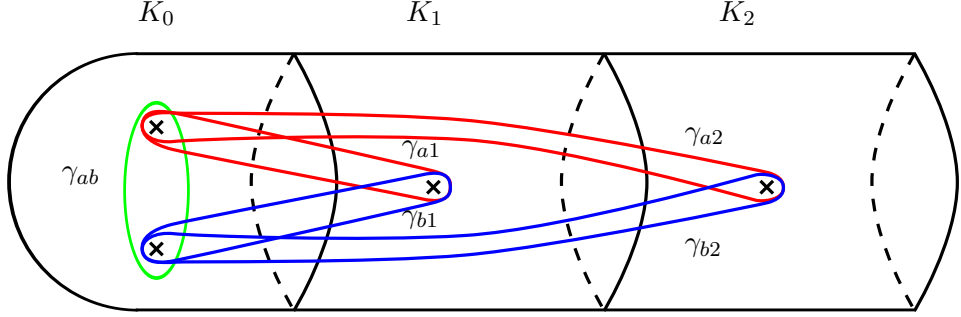
		    \caption{Curves $\gamma_{ab}$ (in green) and $\gamma_{ai}$ and $\gamma_{bi}$ (in red and blue) on the flute surface.}
	    \label{fig:flute}
\end{figure}

Now since the $\phi_{n}$ are compatible we see that in the direct limit we get a map $\phi\colon\cpmcg(\flute) \rightarrow \ZZ$. We claim that $\phi$ extends to a map on $\overline{\cpmcg}(\flute) = \pmcg(\flute)$. We first note that $\phi$ will send any Dehn twist about a curve which does not intersect $K_{0}$ to $0$. Indeed, given any curve $\gamma$ in some $K_{n}$ which does not intersect $K_{0}$ we have two possibilities, either $\gamma$ does not separate $K_{0}$ and $\partial K_{n}$ or it does. Consider the first case. Then $\gamma$ is the boundary of two surfaces, one containing $K_{0}$ and the other, $K'$, not containing $K_{0}$. Thus, $T_{\gamma}$ is the product of Dehn twists about curves contained in $K'$. However, $\phi$ is zero on all of these twists since they live completely outside of $K_{0}$, so $\phi(T_{\gamma})=0$.

Next suppose that $\gamma$ separates $K_{0}$ and $\partial K_{n}$. Then $T_{\gamma}$ can be written in homology as the sum of all curves about pairs of punctures in the component of $K_{n} \setminus \gamma$ which contains $K_{0}$. The only nonzero terms in this sum come in pairs $T_{\gamma_{ai}}$ and $T_{\gamma_{bi}}$ and $\phi(T_{\gamma_{ai}}  T_{\gamma_{bi}}) = 0$. Thus we see that $\phi(T_{\gamma}) = 0$. 

To see that $\phi$ actually extends to the closure of compactly supported pure mapping classes we use the following lemma which is a consequence of the proof of proposition 6.2 in \cite{patelvlamis}.

\begin{lemma}[\cite{patelvlamis}]
    When $S$ is a surface with at most one end accumulated by genus any $f \in \overline{\cpmcg}(S) \setminus \cpmcg(S)$ can be realized as an infinite product of Dehn twists. 
\end{lemma}

Note also that any infinite product of Dehn twists can converge in the pure mapping class group only if the curves about which one twists eventually leave every compact set. To check that $\phi$ extends we realize any given $f \in \overline{\cpmcg}(\flute) \setminus \cpmcg(\flute)$ as an infinite product of Dehn twists about curves which eventually leave every compact set. In particular, they eventually have trivial intersection with $K_{0}$. Thus we see that $\phi$ is non-zero on only finitely many of these Dehn twists so that $\phi(f)$ is finite. Furthermore, if we realize $f$ as two different infinite products then these infinite products must eventually agree on every compact set so that $\phi(f)$ is well-defined and extends. Note that $\phi$ cannot factor through any forgetful map to a finite type surface since forgetting any puncture would make some $\gamma_{ai}$ trivial and $\phi(\gamma_{ai}) \neq 0$ for all $i$. This gives us the following proposition.

\begin{proposition}
    Let $S$ be a surface of genus $0$ with infinitely many punctures. Then there exists a homomorphism from $\pmcg(S)$ to $\ZZ$ that does not factor through a forgetful map to a sphere with finitely many punctures. 
\end{proposition}

\begin{proof}
    Given any such $S$
    we have maps 
    \[\pmcg(S) \xrightarrow{\mathcal{F}} \pmcg(\flute) \xrightarrow{\phi} \ZZ
    \]
 where $\mathcal{F}$ is a forgetful map and $\phi$ is the homomorphism constructed above. Now $\phi \circ \mathcal{F}$ gives a homomorphism from $\pmcg(S)$ to $\ZZ$ which cannot factor through a forgetful map to a sphere with finitely many punctures. 
 \end{proof}

In our construction of $\phi$ it was only important that $\phi(T_{\gamma_{ai}} T_{\gamma_{bi}}) = 0$. Letting $\phi_{i}$ be defined the same way as $\phi$ except set $\phi_{i}(T_{ai})$ and $\phi_{i}(T_{bi})$ to be zero, we get an infinite family of maps $\{\phi_{i}\}$. In fact, for any subset, $A$, of positive integers with the complement of $A$ infinite we can define $\phi_A$ to be zero at $a_i,b_i$ when $i \in A$. This gives an uncountable collection of $A$ where $\phi_A$ will not factor through forgetful maps to finite type surfaces. Thus we have

\begin{corollary}
    If $S$ is a genus $0$ surface with infinitely many punctures then there is an uncountable family of homomorphisms to $\ZZ$ which do not factor through a forgetful map to a finite type surface.
\end{corollary}

These two results together give Theorem \ref{genus0}.

\end{document}